\documentclass[a4paper, 10pt]{amsart}

\usepackage[T1]{fontenc}
\usepackage{latexsym, amssymb, amsfonts}
\usepackage{amsthm, amsmath}
\usepackage[svgnames]{xcolor}
\usepackage{mathrsfs}
\usepackage{leftidx}
\usepackage{enumerate}
\usepackage{calc}
\usepackage{pdfsync}
\usepackage{graphicx}
\usepackage{stmaryrd}

\usepackage{setspace}

\input xy
\xyoption{all}
\xyoption{ps}
\SelectTips{cm}{}
\CompileMatrices

\newdir{ >}{{}*!/-5pt/\dir{>}}

\usepackage{hyperref}

\newcommand{\mono}{\rightarrowtail}

\DeclareMathOperator{\Ker}{Ker}
\DeclareMathOperator{\Coim}{Coim}

\DeclareMathOperator{\coim}{coim}

\DeclareMathOperator{\coker}{coker}

\newcommand{\mat}[1]{\left[\begin{smallmatrix} #1 \end{smallmatrix}\right]}

\newcommand{\scrA}{\mathscr{A}}

\newcommand{\scrE}{\mathscr{E}}

\newcommand{\calA}{\mathcal{A}}

\newcommand{\calD}{\mathcal{D}}

\newcommand{\calF}{\mathcal{F}}

\makeatletter
\def\@strippedMR{}
\def\@scanforMR#1#2#3\endscan{%
  \ifx#1M\ifx#2R\def\@strippedMR{#3}%
  \else\def\@strippedMR{#1#2#3}%
  \fi\fi}
\renewcommand\MR[1]{\relax\ifhmode\unskip\spacefactor3000 \space\fi
  \@scanforMR#1\endscan
  MR\MRhref{\@strippedMR}{\@strippedMR}}
\makeatother 

\makeatletter
\newcommand\@dotsep{4.5}
\def\@tocline#1#2#3#4#5#6#7{\relax
  \ifnum #1>\c@tocdepth 
  \else
    \par \addpenalty\@secpenalty\addvspace{#2}%
    \begingroup \hyphenpenalty\@M
    \@ifempty{#4}{%
      \@tempdima\csname r@tocindent\number#1\endcsname\relax
    }{%
      \@tempdima#4\relax
    }%
    \parindent\z@ \leftskip#3\relax \advance\leftskip\@tempdima\relax
    \rightskip\@pnumwidth plus1em \parfillskip-\@pnumwidth
    #5\leavevmode\hskip-\@tempdima #6\relax
    \leaders\hbox{$\m@th
      \mkern \@dotsep mu\hbox{.}\mkern \@dotsep mu$}\hfill
    \hbox to\@pnumwidth{\@tocpagenum{#7}}\par
    \nobreak
    \endgroup
  \fi}
\makeatother 

\newtheoremstyle{mythm}
 {6pt}
 {6pt}
 {\itshape}
 {}
 {\scshape}
 {.}
 {.5em}
 {}%

\newtheoremstyle{mydef}
 {6pt}
 {6pt}
 {}
 {}
 {\scshape}
 {.}
 {.5em}
 {}%
 
\makeatletter
\renewenvironment{proof}[1][\proofname]{\par
  \pushQED{\qed}%
  \normalfont \topsep6\p@\@plus6\p@\relax
  \trivlist
  \item[\hskip\labelsep
        \scshape
    #1\@addpunct{.}]\ignorespaces
}{%
  \popQED\endtrivlist\@endpefalse
}
\makeatother

\swapnumbers

\theoremstyle{mythm}

\newtheorem*{Lemnn}{Lemma}
\newtheorem*{Cornn}{Corollary}
\newtheorem*{Propnn}{Proposition}

\theoremstyle{mydef}

\newtheorem*{Remnn}{Remark}

\usepackage{enumitem}

\title[Partially abelian exact categories]%
{Remarks on partially abelian exact categories}
\author{Theo B\"uhler}
\address{}
\email{math@theobuehler.org}
\thanks{}
\date{February 25, 2013}

\newcommand{\AIC}{\ensuremath{(\mathrm{AIC})}}
\newcommand{\AICO}{\ensuremath{\AIC^\circ}}

\begin{document}

\begin{abstract}
  The purpose of this short and elementary note is to identify some classes of
  exact categories introduced in L. Previdi's thesis. Among other
  things we show: 
  \begin{enumerate}[label=(\roman{*})]
    \item
      An exact category is partially abelian exact if and only if it is
      abelian. 
    \item
      An exact category satisfies the axioms $\AIC$ and $\AICO$
      if and only if it is quasi-abelian in the sense of J.-P. Schneiders.
    \item
      An exact category satisfies $\AIC$ if and only if it is an additive
      category of the type considered by G. Laumon in his work on
      derived categories of filtered $\calD$-modules.
  \end{enumerate}
  In all of the above classes all morphisms have kernels and coimages
  and the exact structure must be given by all kernel-cokernel
  pairs. 
\end{abstract}

\maketitle

\section{Preliminary observations}

For every morphism $f \colon A \to B$ in an additive category there is
an associated automorphism $\mat{1 & 0 \\ f & 1}$ of $A \oplus
B$. It yields an isomorphism of short sequences
\[
\xymatrix{
  A \ar@{ >->}[r]^-{\mat{1 \\ 0}} \ar@{=}[d] &
  A \oplus B \ar@{->>}[r]^-{\mat{0 & 1}} 
  \ar[d]^{\mat{1 & 0 \\ f & 1}}_{\cong} &
  B \ar@{=}[d] \\
  A \ar@{ >->}[r]^-{\mat{1 \\ f}} &
  A \oplus B \ar@{->>}[r]^-{\mat{-f & 1}} &
  B.
}
\]
The morphism $\mat{1 \\ f}$ can be interpreted as the
graph of $f$. It has a left inverse given by the projection onto the
first coordinate, while composition with the
projection onto the second coordinate yields the
morphism $f = \mat{0 & 1}\mat{1 \\ f}$. In particular,
every morphism in an additive category is the composition of a split
monic followed by a split epic.

The following lemma simply states that the intersection of the graph of
$f$ with the graph of the zero morphism coincides with the kernel of
$f$. The verification is straightforward.

\begin{Lemnn}
  Let $f \colon A \to B$ be a morphism in an additive category.
  Consider the diagram of solid arrows below
  \[
  \xymatrix{
    K \ar@{.>}[r]^{k} \ar@{.>}[d]_{k} \ar@{}[dr]|{\mathrm{PB}} &
    A \ar@{ >->}[d]^{\mat{1 \\ f}} \ar[r]^{f} &
    B \ar@{=}[d] \\
    A \ar@{ >->}[r]^-{\mat{1 \\ 0}} & 
    A \oplus B \ar@{ ->>}[r]^-{\mat{0 & 1}} & 
    B.
  }
  \]
  The pull-back square on the left exists if
  and only if $f$ has a kernel $k \colon K \mono A$. \qed
\end{Lemnn}
\if{0}
\begin{proof}
  This is a straightforward diagram chase.
  Suppose $f$ has a kernel $k \colon K \to A$. Set $i = k$ and observe that the
  diagram on the left commutes. Let us check that it is a pull-back.
  If $\alpha,\beta \colon X \to A$ are such that $\mat{1 \\ f} \alpha
  = \mat{1 \\ 0}\beta$ then $\alpha = \beta$ and $f\alpha = 0$. The
  latter equation implies that $\alpha = k\alpha'$ for a unique
  $\alpha' \colon X \to K$ and thus the left hand square is a
  pull-back.

  Conversely, suppose the pull-back diagram exists. Since pull-backs
  of monics are monic, $i$ and $k$ are monic. Moreover, $\mat{1 \\
    f}k = \mat{1 \\ 0}i$ implies $k = i$ and $fk = 0$. 
  To check that $k$ is a kernel of $f$, suppose $\alpha
  \colon X \to A$ is such that $f\alpha = 0$. Then $\mat{1 \\ f}\alpha
  = \mat{1 \\ 0}\alpha$ and the pull-back property of the left hand
  square yields a unique $\alpha' \colon X \to K$ such that $\alpha =
  k \alpha'$. 
\end{proof}
\fi

\begin{Cornn}
  An additive category admits pull-backs of all pairs of split monics if
  and only if it has kernels (if and only if it is finitely complete). \qed
\end{Cornn}

\section{Intrinsic characterization of exact categories satisfying
\texorpdfstring{$\AIC$}{AIC}}

Previdi's additional axiom $\AIC$ for exact categories
\cite[3.17]{previdi-sato} states that pull-backs of admissible
monics along admissible monics exist and are admissible monics. Since
split monics with cokernels are always admissible, we infer from $\AIC$ and the
previous section that the underlying additive category must have
kernels. Moreover, $\AIC$ and the lemma imply that all kernels are
admissible monics.

Suppose $f$ has a cokernel $c$. Then $\ker{c}$ is an admissible monic
and $c = \coker{\ker{c}}$ implies that $c$ is an admissible
epic. Therefore the class of cokernels is equal to the class of
admissible epics.

This discussion establishes the necessity part of the following
characterization of exact categories satisfying $\AIC$:

\begin{Propnn}
  $(\scrA,\scrE)$ is an exact category satisfying $\AIC$ if and only if:
  \begin{enumerate}[label=(\roman{*})]
    \item
      Every morphism in $\scrA$ has a kernel.
    \item
      The push-out of a kernel along an arbitrary morphism exists and
      is a kernel.
    \item 
      Cokernels are stable under pull-backs along arbitrary
      morphisms.
    \item
      All kernels are admissible monics and all cokernels are admissible
      epics and $\scrE$ is the class of all kernel-cokernel pairs in $\scrA$. 
  \end{enumerate}
  In particular, if $\scrA$ admits an exact
  structure satisfying $\AIC$, it is the unique maximal exact
  structure on $\scrA$. Moreover, kernels are stable under pull-backs along arbitrary
  morphisms.
\end{Propnn}
\begin{proof}
  Suppose $(\scrA,\scrE)$ satisfies (i), (ii), (iii) and (iv). 
  Note that (ii) implies that every kernel has a cokernel and with
  (i) this shows that every morphism has a coimage
  (form the push-out of the kernel along zero).
  
  In order to check that $\scrE$ is an exact structure,
  Keller~\cite[A.1]{MR1052551
  }
  shows that it is enough to verify that the composition of kernels is again a
  kernel.

  To this end, let $(i,p)$ and 
  $(j,q)$ be short exact such that $ji$ is defined.
  We need to prove that $ji$ is a kernel. Form the push-out under $j$
  and $p$ to obtain the diagram
  \[
  \xymatrix{
    A' \ar@{ >->}[r]^{i} \ar@{=}[d] &
    A \ar@{->>}[r]^-{p} \ar@{ >->}[d]^{j}
    \ar@{}[dr]|{\mathrm{PO}} &
    A'' \ar@{ >.>}[d]^{j''}  \\
    A' \ar[r]^{ji} & 
    B \ar@{.>}[r]^-{r} \ar@{->>}[d]^{q} & 
    B'' \ar@{.>}[d]^{q''} \\
    & C \ar@{=}[r]
    & C
  }
  \]
  where $q''$ is uniquely determined by $q''r = q$ and $q''j'' =
  0$. Since $q'' = \coker{j''}$ we have $(j'',q'') \in \scrE$. Using
  the push-out property one checks that $r = \coker(ji)$ and finally
  one checks that $ji = \ker r$ so that $(ji,r) \in \scrE$.

  It remains to verify $\AIC$. Consider the following diagram in which
  $(b,b') \in \scrE$ and $f$ is arbitrary:
  \[
  \xymatrix{
    A' \ar@{ >.>}[r]^{a} \ar@{.>}[d]^{f'} & 
    A \ar@{.>>}[r]^{a'} \ar[d]^{f} &
    A'' \ar@{.>}[d]^{f''} \\
    B' \ar@{ >->}[r]^{b} &
    B \ar@{->>}[r]^{b'} &
    B''.
  }
  \]
  Define $(a,a') = (\ker(b'f),\coim(b'f))$. Since $b'fa = 0$, there
  are unique $f'$ and $f''$ making the diagram commutative.
  If $\alpha \colon X
  \to A$ and $\beta' \colon X \to B'$ are such that $f \alpha = b\beta'$ 
  then $b'f\alpha = b'b\beta = 0$ and since $a = \ker(b'f)$,
  there is a unique morphism $\alpha' \colon X
  \to A'$ such that $a\alpha' = \alpha$. Because $b$ is monic we
  conclude $\beta' = f'\alpha'$. Thus, the left hand square is
  a pull-back and $\AIC$ follows.
\end{proof}

\begin{Remnn}
  Exact categories satisfying $\AIC$ are studied in
  \cite[(1.3.0), p.160ff]{MR726427
  }.
\end{Remnn}

\begin{Remnn}
  Since admissible monics have cokernels, every morphism has in
  addition a coimage.
  \[
  \vcenter{
    \xymatrix{
      & A \ar[rr]^{f} \ar@{->>}[dr]_{p=\coim{f}} & & B \\
      \Ker{f} \ar@{ >->}[ur]^{\ker{f}} & & \Coim{f} \ar@{^(->}[ur]_{m}
    }
  }
  \qquad \qquad
  \vcenter{
    \xymatrix{
      Y \ar[r]^-{y} \ar@{->>}[d]^-{q} \ar@{}[dr]|{\mathrm{PB}}
      & A \ar@{->>}[d]^-{p} \\
      X \ar[r]^-{x} & \Coim{f}.
    }
  }
  \]
  The morphism $m\colon \Coim{f} \to B$ is monic: if $mx = 0$,
  form the pull-back over $x$ and $p$, notice that $fy = 0$ so that $y$
  factors through $\Ker{f}$. Therefore $0 = py = xq$, so $x = 0$ since
  $q$ is epic. [A variant of this argument can be used to give a
  direct proof of \cite[Lemma~3.18]{previdi-sato}.]
\end{Remnn}

\section{Exact categories satisfying \texorpdfstring{$\AIC$}{AIC}
  and \texorpdfstring{$\AICO$}{AIC°}}

The characterization of exact categories satisfying $\AIC$ in the
previous section is very nearly self-dual. The only missing piece is the existence of
cokernels. Thus:

\begin{Propnn}
  $(\scrA, \scrE)$ is an exact category with $\AIC$ and $\AICO$ if
  and only if:
  \begin{enumerate}[label=(\roman{*})]
    \item
      Every morphism in $\scrA$ has a kernel and a cokernel.
    \item
      Kernels are stable under push-outs along arbitrary morphisms.
    \item 
      Cokernels are stable under pull-backs along arbitrary morphisms.
    \item
      $\scrE$ is the class of all kernel-cokernel pairs in $\scrA$. 
  \end{enumerate}
  These are precisely the quasi-abelian categories of 
  Schneiders~\cite{MR1779315
  } and the almost abelian categories of Rump~\cite{MR1856638
}. \qed
\end{Propnn}

\section{Partially abelian exact categories are abelian}

An exact category is called \emph{partially abelian exact} if every morphism
which is the composition of an admissible monic followed by an admissible
epic has a factorization as an admissible epic followed by an
admissible monic:
\[
\xymatrix{
  & C \ar@{->>}[dr]^{e} \\
  A \ar@{ >->}[ur]^{m} \ar@{.>>}[dr]_{\exists p} \ar[rr]^f & & B \\
  & I \ar@{ >.>}[ur]_{\exists i}
}
\]
that is, $f= em$ with $m$ admissible monic and $e$ admissible epic
implies $f= ip$ with $p$ admissible epic and $i$ admissible
monic. This factorization shows that $\ker{f} = \ker{p}$ and
$\coker{f} = \coker{i}$ and thus $I$ is both coimage and image of $f$.

From the first section we know that every morphism $f \colon A \to B$ in
an additive category can be written as the composition of a split
monic followed by a split epic $f = \mat{0 & 1} \mat{1 \\ f}$.
Thus, in a partially abelian exact category every morphism $f$ factors
as $f = ip$ via admissible epic $p$ followed by an admissible monic $i$ 
and hence every $f$ has a kernel and a cokernel.

Consider the factorization $f = ip$ of a monic $f$ into an admissible
epic followed by an admissible monic. Then $p$ is monic
and since monic cokernels in an additive category are
isomorphisms, $p$ is an
isomorphism from $f$ to the admissible monic $i$, so $f$ is an
admissible monic, and hence it is a kernel. Dually, every epic is an admissible
epic and hence it is a cokernel.


\begin{Propnn}
  A category is partially abelian exact if and only if it is
  abelian. \qed
\end{Propnn}

\section{Discussion of Previdi's Theorem 3.24 and Proposition 3.22}

The facts that for an exact category we have 
\begin{gather*}
  \text{partially abelian exact} \iff \text{abelian} \\
  \text{$\AIC$ and $\AICO$} \iff \text{quasi-abelian}
\end{gather*}
contradict Theorem~3.24 in Previdi~\cite{previdi-sato}. The reason is
that the proof of the theorem is based on the incorrect Proposition~3.22. 

\begin{Propnn}
  Consider the additive category $\mathsf{Ban}$ of Banach spaces and
  bounded linear maps. Equip it   with the usual maximal exact structure consisting of the short
  sequences whose underlying sequence of vector spaces is exact. Then
  $\mathsf{Ban}$ satisfies axioms $\AIC$ and $\AICO$.
  On the other hand, $\mathsf{Ban}$ is not partially abelian exact.
\end{Propnn}
\begin{proof}
  One could appeal to the known fact that $\mathsf{Ban}$ is
  quasi-abelian and the proposition in section~3,
  but it is just as easy to verify $\AIC$ and $\AICO$ directly:

  Admissible monics are precisely the injective maps with
  closed range. Every admissible monic is isomorphic to the inclusion
  of a closed subspace. Given two closed subspaces $U$ and $V$ of a
  Banach space $E$, their pull-back is the intersection $U \cap V$
  which is clearly closed in both $U$ and $V$, and
  $(\operatorname{AIC})$ follows.

  Admissible epics are precisely the surjective maps. The push-out of
  two morphisms in $\mathsf{Ban}$ is a quotient of their push-out in the category of
  vector spaces, so $(\operatorname{AIC})^\circ$ follows from the facts
  that $\mathsf{Vect}$ satisfies $(\operatorname{AIC})^\circ$ and that the
  composition of two surjective maps is surjective.

  Since there are morphisms in $\mathsf{Ban}$ which do not have closed
  range, it cannot be partially abelian exact.
\end{proof}

To see what goes wrong in the proof of Proposition~3.22, consider
a bimorphism (monic-epic)  $f \colon A \to B$ which is not an
isomorphism, e.g. the inclusion $\ell^1 \subset \ell^2$ in
$\mathsf{Ban}$.
In the following diagram the first two rows and columns are split exact:
\[
\xymatrix{
  0 \ar@{ >->}[r] \ar@{ >->}[d] \ar@{}[dr]|{\text{PB}} &
  A \ar@{->>}[r]^{1} \ar@{ >->}[d]^{\mat{1 \\ f}} &
  A \ar[d]^{f} \\
  A \ar@{ >->}[r]^-{\mat{1 \\ 0}} \ar@{->>}[d]^{1} &
  A \oplus B \ar@{->>}[r]^-{\mat{0 & 1}} \ar@{->>}[d]_-{\mat{-f & 1}}&
  B \\
  A \ar[r]^-{-f} & B 
}
\]
Since $f$ is monic we have $0 = \Ker{f}$. It follows from the lemma in
the first section that the upper left corner is a
pull-back. Dually, the push-out of the two morphisms in the lower
right corner is $0$.

The Quillen embedding preserves pull-backs, so
the morphism $f$ is still monic in the abelian envelope
$\calF$: the left square of a morphism of short exact sequences in an
abelian category is a pull-back if and only if the rightmost arrow is
monic.

Of course, $f$ cannot be the kernel of $B \to 0$ (neither in the
abelian envelope $\calF$ nor in $\calA$) since it is
not an isomorphism. This gives an explicit example showing that the
Quillen embedding preserves neither epics ($f$ is epic in $\calA$ but
it is not epic in $\calF$), nor
cokernels (the cokernel of $f$ is no longer $B \to 0$), nor push-outs
(not even those involving only admissible epics), 
as seems to be assumed in the proof of Previdi's Proposition~3.22. To
reiterate: \emph{it does not follow that} $t \in \calA$, contrary to
what is claimed in the proof.

Note that the above diagram can be substituted as diagram $(3.21)$ by
adding the cokernel $e'' \colon B \to t \in \calF$ of $f$ in the abelian envelope and the
corresponding factorization $j''$ of $\mat{0 & e''}\colon A \oplus B
\to t$ over $\mat{-f & 1}$.

\appendix
\section*{Postscriptum}

This is an unmodified version of a note sent to Braunling, Groechenig and
Wolfson.
It was written late 2012 or early 2013 when I answered some questions that
arose in their work on Tate objects in exact categories.
At that point I also informed Previdi of these results.
While others have since rediscovered variants of these ideas, there
is still no citable reference for this material, so I made this publicly
available on Braunling's request.

Br\"ustle pointed out to the author that Hassoun and Roy~\cite{HR19}
independently introduced $AI$-categories 
(pre-abelian exact categories satisfying $\AIC$).
Br\"ustle, Hassoun, Shah, Tatar and Wegner showed that $AI$-categories
and quasi-abelian categories are the same, see~\cite[Theorem 1.3]{BHT20}.
This is a variant of the results in sections 2 and 3 of this note.
Readers who found this simple note interesting will find more
in \cite{BHT20} and the references therein.

\bibliographystyle{amsalpha}
\bibliography{bibliography}

\end{document}